\documentclass[12pt]{article}

 \usepackage{amsmath,amssymb,amscd,amsthm,esint}

\usepackage{graphics,amsmath,amssymb,amsthm,mathrsfs}

\usepackage{graphics,amsmath,amssymb,amsthm,mathrsfs,amsfonts}
\usepackage{sidecap}
\usepackage{float}
\usepackage{extarrows}
\usepackage{booktabs}
\usepackage{verbatim}
\usepackage{hyperref}
\usepackage[usenames,dvipsnames]{xcolor}

\newcommand{\rn}[1]{\mathbb{R}^{#1}}

\newcommand{\diverg}[1]{\mbox{div}\,#1\,}

\newtheorem{thm}{Theorem}[section]
\newtheorem{lemma}[thm]{Lemma}
\newtheorem{cor}[thm]{Corollary}

\theoremstyle{definition}
\newtheorem{remark}[thm]{Remark}

\def\XXint#1#2#3{{\setbox0=\hbox{$#1{#2#3}{\int}$}
         \vcenter{\hbox{$#2#3$}}\kern-.5\wd0}}

\def\R{\mathbb{R}}

\def\e{\varepsilon}

\numberwithin{equation}{section}

\begin{document}

\bibliographystyle{amsplain}

\title{Neumann Problems for the Stokes Equations\\  in Convex Domains}

\author{Jun Geng  \thanks{Supported in part by NNSF grant 12371096.}
\qquad
Zhongwei  Shen \thanks{Supported in part by NSF grant DMS-2153585.}}
\date{}

\maketitle

\begin{abstract}

This paper studies the Neumann boundary value problems for the Stokes equations in a convex domain in $\mathbb{R}^d$.
We obtain  nontangential-maximal-function estimates in $L^p$ and $W^{1, p}$ estimates
 for $p$ in certain ranges depending on $d$.
These ranges are larger than the known ranges for Lipschitz domains.
The proof relies on a $W^{2, 2}$ estimate for the Stokes equations in convex domains.


\medskip


\medskip


\end{abstract}


\section{Introduction}

Let $\Omega$ be a bounded convex domain in $\rn{d}$, $d\geq 2$,
and $n$  the outward unit normal to
$\partial\Omega$. In this paper we are interested in the $L^p$ Neumann boundary value problem for
the Stokes equations,
\begin{equation}\label{stokesSystem}
\left\{
\aligned
        -\Delta u +  \nabla \phi  &=0  &  \quad & \mbox{ in } \Omega, \\
        \text{\rm div} (u) &  = 0  & \quad & \mbox{ in } \Omega, \\
        \frac{\partial u}{\partial \nu} & =g  & \quad &  \mbox{ on }
        \partial\Omega, \\
        (\nabla u)^*, (\phi)^*  & \in L^p(\partial\Omega),
    \endaligned
    \right.
    \end{equation}
where  $g\in L^p(\partial\Omega; \mathbb{R}^d)$,
\begin{equation}\label{co}
\frac{\partial u}{\partial \nu}=\frac{\partial u}{\partial n}-\phi n,
\end{equation}
and $(\nabla u)^*, (\phi)^*$ denote  the nontangential maximal functions
of $\nabla u, \phi$, respectively.
 The  Neumann condition $\frac{\partial  u}{\partial \nu} =g$ in \eqref{stokesSystem}  is taken in the sense of
nontangential convergence;
i.e.
$$
\lim_{\substack{y \rightarrow x \\ y\in \Gamma(x)}} \Big\{ \nabla u_i (y) \cdot n(x)-  \phi (y) n_i(x)  \Big\}
= g_i(x)\quad  \mbox{ for a.e. } x\in \partial\Omega
$$
and  $1\le i\le d$,
 where
$\Gamma(x) = \left\{  y \in \Omega:
|y-x|<C_0 \, \mbox{dist}(y,\partial\Omega)\right\}$.
Let $L_0^p(\partial \Omega; \R^d)$  denote the subspace of $L^p(\partial \Omega; \R^d)$ of functions $g$ with $\int_{\partial \Omega} g=0$.
We say that the $L^p$ Neumann
problem (\ref{stokesSystem})  is uniquely solvable if
given any $g\in L^p_0(\partial\Omega; \R^d)$, there exists a
unique  $(u, \phi) $ (up to constants for $u$), satisfying (\ref{stokesSystem}), and  the
solution $(u, \phi)$ satisfies
\begin{equation}\label{1.2}
\|(\nabla u)^*\|_{L^p(\partial\Omega)} + \| (\phi)^* \|_{L^p(\partial\Omega)}
\leq
C\|g\|_{L^p(\partial\Omega)},
\end{equation}
where $C$ depends only on $d, p$ and $\Omega$.

The following is one of  the main results of the paper.

\begin{thm}\label{NeumannProblem}
    Let $\Omega$ be a bounded convex domain in $\rn{d}$, $d\geq
    2$.  Then  the $L^p$ Neumann problem  \eqref{stokesSystem} is uniquely solvable  if
    \begin{equation}\label{p}
    \left\{
    \aligned
         1< &p<\infty  &  \qquad & \mbox{ for } d=2, \\
         1< & p<4 +\e  & \quad & \mbox{ for } d=3, \\
         \frac{2(d-1)}{d+1}-\e< & p<\frac{2(d-1)}{d-2}  +\e  & \quad & \mbox{ for } d\geq 4,
    \endaligned
    \right.
    \end{equation}
     where $\varepsilon=\e(\Omega)>0$.
\end{thm}

\begin{remark}
Note that a convex domain is a Lipschitz domain.
For a bounded Lipschitz domain in $\mathbb{R}^d$, the $L^p$ Neumann problem for the Stokes equations is uniquely solvable for
\begin{equation}\label{Lp-p-Lipschitz}
\left\{
\aligned
1< & p< 2+\e  & \quad & \text{ if } d=2 \text{ or } 3,\\
\frac{2(d-1)}{d+1}-\e< & p<2+\e & \quad & \text{ if } d\ge 4.\\
\endaligned
\right.
\end{equation}
The case $2-\e< p< 2+\e$ was proved by the method of layer potentials in \cite{FKV-1988},
while the remaining case for $p<2$ was obtained in \cite{MW-2012}.
The upper bound $2+\e$ in \eqref{Lp-p-Lipschitz} is sharp for a general Lipschitz domain.
Theorem \ref{NeumannProblem}  improves the upper bounds of $p$ in the case of  convex domains.
It is not known whether  the lower and upper bounds of $p$ in \eqref{p} for $d\ge 4$
as well as the upper bound for $d=3$  are sharp for convex domains.
\end{remark}

In this paper we also  study the $W^{1,p}$ estimate for the inhomogeneous Stokes equations in a convex domain
with the Neumann boundary condition. Let $1< p< \infty$.
For any $f\in L^p(\Omega; \R^{d\times d} )$, $F\in L^p(\Omega; \R^d)$ and $g\in B^{-1/p,p}(\partial\Omega; \R^d)$,
we seek a  solution $(u,\phi)\in W^{1,p}(\Omega; \R^d)\times L^p(\Omega)$ such that
\begin{equation}\label{inhomogeneousStokes}
\left\{
\aligned
       - \Delta u + \nabla \phi  & =\text{div}(f)+F & \quad &  \mbox{ in } \Omega, \\
        \text{\rm div}(u)   &= 0   & \quad & \mbox{ in }\Omega, \\
        \frac{\partial u}{\partial \nu} & =-n\cdot f +g  & \quad &  \mbox{ on }
        \partial\Omega,
    \endaligned
    \right.
    \end{equation}
holds in the weak sense; i.e.,
\begin{equation}\label{weak-NP}
\int_\Omega \nabla u\cdot \nabla \varphi
- \int_\Omega \phi\,  \text{\rm div}(\varphi)
=-\int_\Omega  f\cdot \nabla \varphi
+\int_\Omega F \cdot \varphi
+\langle g, \varphi \rangle_{B^{-1/p, p}(\partial\Omega)\times B^{1/p, p^\prime}(\partial\Omega)}
\end{equation}
for any $\varphi \in C^\infty(\R^d; \R^d)$, and $\int_\Omega u \cdot \nabla \psi=0$
for any $\psi \in C_0^\infty(\Omega)$.
Here, $B^{1/p, p^\prime}(\partial\Omega)$ denotes the Besov space of order $1/p$ on $\partial\Omega$ and
$B^{-1/p, p}(\partial\Omega)$ the dual of $B^{1/p, p^\prime}(\partial\Omega)$.
By taking $\varphi=e_j$ in \eqref{weak-NP},
 we obtain a compatibility condition,
\begin{equation}\label{comp}
\int_\Omega F \cdot e_j
=- \langle g, e_j \rangle_{B^{-1/p, p}(\partial\Omega) \times B^{1/p, p^\prime}(\partial\Omega)}
\end{equation}
for $1\le j \le d$.
In the case $p=2$ and $\Omega$ is Lipschitz, it is  known that  the solution $(u, \phi)$ exists and is unique (up to constants for $u$). Moreover, we have the energy estimate,
\begin{equation}\label{energy-1}
    \|\nabla u\|_{L^2(\Omega)}+\|\phi\|_{L^2(\Omega)} \leq C\left\{ \|f\|_{L^2(\Omega)}+\|F\|_{L^2(\Omega)}+\|g\|_{B^{-1/2,2}(\partial\Omega)} \right\},
    \end{equation}
    where $C$ depends on $\Omega$.

\begin{thm}\label{W1p}
    Let $\Omega$ be a bounded convex domain in $\rn{d}$, $d\geq
    2$. Suppose that
      \begin{equation}\label{w1p-p-convex}
    \left\{
    \aligned
         1< & p<\infty   & \quad & \mbox{ if } d=2, \\
          \big| \frac{1}{p}-\frac12 \big| & < \frac{1}{d} + \e   & \quad & \mbox{ if } d\geq 3,
    \endaligned
    \right.
    \end{equation}
    where $\varepsilon=\e(\Omega)>0$.
    Then for any $F\in L^p(\Omega; \R^d), f\in L^p(\Omega; \R^{d\times d} )$,
     and $g\in B^{-1/p,p}(\partial\Omega; \R^d)$ satisfying the compatibility condition \eqref{comp},
     there exists a unique $(u,\phi)\in W^{1,p}(\Omega; \R^d)\times L^p(\Omega)$(up to  constants for $u$) such that (\ref{inhomogeneousStokes})  holds in the weak sense.
   Moreover, the solution $(u,\phi)$ satisfies the estimate,
    \begin{equation}\label{1.3}
    \|\nabla u\|_{L^p(\Omega)}+\|\phi\|_{L^p(\Omega)} \leq C\left\{ \|f\|_{L^p(\Omega)}+\|F\|_{L^p(\Omega)}+\|g\|_{B^{-1/p,p}(\partial\Omega)} \right\},
    \end{equation}
    where $C$ depends on $d, p$ and $\Omega$.
\end{thm}

\begin{remark}
If $\Omega$ is a bounded Lipschitz domain in $\mathbb{R}^d, d\geq 2$,  the $W^{1,p}$ estimate \eqref{1.3}  holds for $p$ satisfying
\begin{equation}\label{w1p-p-Lipschitz}
\big| \frac{1}{p}-\frac12 \big|< \frac{1}{2d} +\e.
\end{equation}
This was established  in  \cite{GK-2009}
 for the Stokes equations with the Dirichlet condition.
  The method used in \cite{GK-2009} extends readily to  the Neumann problem.
However, it is not known whether the $W^{1, p}$ estimate holds for the Stokes equations with the Dirichlet condition
 in a convex domain for $p$ satisfying \eqref{w1p-p-convex}.
 The approach  used in this paper does not seem to extend to the Dirichlet problem.
\end{remark}

During the last decades, considerable progress has been made in the study of  second order elliptic equations and systems in non-smooth domains.
For Laplace's equation $\Delta u=0$ in a Lipschitz domain, it is well-known that the $L^p$ Dirichlet and $L^p$ Neumann problems are uniquely solvable with sharp ranges $2-\varepsilon < p \leq \infty$ and $1<p< 2+\e$ respectively, where $\varepsilon$ depends on the Lipschitz character of $\Omega$ \cite{Kenig-book}. 
For Laplace's equation in a  convex domain,
the $L^p$ Dirichlet problem is solvable for $1<p\le \infty$, while the $L^p$ Neumann
problem is solvable for $1<p<\infty$.
The Dirichlet case follows  readily by using the maximum principle.
The Neumann case, which was studied  in \cite{Kim-Shen-2008, Mazya-2009, Geng-Shen-2010, Lewis-2014},   is much more involved.

We now move on to the Stokes equations. As mentioned earlier,
the $L^2$ Dirichlet and Neumann problems in a Lipschitz domain  were solved in \cite{FKV-1988} by the method of layer potentials.
Due to the lack of maximum principles and De Giorgi-Nash-Moser estimates,
the case $p\neq 2$ has been very challenging..
It is known that the $L^p$ Dirichlet problem is solvable for $2-\e< p \le \infty$ if $d=2$ or $3$ \cite{Shen-1995},
and for $2-\e< p< \frac{2(d-1)}{d-3} +\e$ if $d\ge 4$ \cite{Kilty-2009}.
It is also proved  that the  $L^p$ Neumann problem is solvable for $1< p< 2+\e$ if $d=2$ or $3$, and for $\frac{2(d-1)}{d+1}-\e< p< 2+\e$ if  $d\ge 4$ \cite{MW-2012}.
 The methods used in \cite{Kilty-2009, MW-2012} for $d\ge 4$ are similar to those developed  in \cite{Shen-2006, Shen-2007} for
 second-order elliptic systems.
  The question of the sharp range of $p$'s  remains open for both the Dirichlet and Neumann problems in the case $d\ge 4$,
  as in the case of second-order elliptic systems.

Finally, we  describe the proofs  of Theorems \ref{NeumannProblem} and \ref{W1p}.
For $x_0\in \partial\Omega$ and $0< r<r_0$,
let $D(x_0, r)=  B(x_0, r)\cap \Omega$ and $I(x_0, r) =B(x_0, r) \cap \partial\Omega$.
Our approaches  to both theorems are based on  a new $W^{2,2}$ estimate for the Stokes equations on a convex domain,
\begin{equation}\label{1.8}
 \int_{D(x_0,r)} \left(  |\nabla^2 u|^2 +|\nabla \phi|^2 \right) \,dx \leq
        \frac{C}{r^2}\int_{D(x_0,2r)}\left(  |\nabla u|^2 +|\phi|^2 \right) \,dx,
\end{equation}
where  $\Omega$ is convex and $u$ satisfies
\begin{equation}\label{local-sol}
\left\{
\aligned
-\Delta u+\nabla \phi & =0 & \quad  & \text{ in }  D(x_0,4r),\\
 ~\text{\rm div}(u)  & = 0 & \quad & \text{ in } D(x_0,4r), \\
 \frac{\partial u}{\partial\nu} & =0 & \quad & \text{ on }  I(x_0,4r).
 \endaligned
 \right.
\end{equation}
To prove \eqref{1.8}, we use a well-known integral identity (see Lemma \ref{convex})  with
 $v=w_\beta=\big(\{ \partial_k u_\beta-\delta_{k\beta}\phi\}\xi\big)_{k=1}^d$
for $1\leq \beta \leq d$ and $\xi \in C_0^\infty(B(x_0, 2r))$, to obtain
$$
        \int_{\Omega} |{\rm div}(v)|^2\,dx\geq \sum_{i,j=1}^d\int_{\Omega}\frac{\partial v_i}{\partial x_j}\frac{\partial v_j}{\partial x_i}\,dx,
$$
which yields \eqref{1.8} by using the Stokes equations as well as the Cauchy inequality.
We point out that a similar observation was used in \cite{Tolksdorf-2020} to study the $L^p$ resolvent estimates for the
Stokes equations with Neumann conditions.

With the estimate \eqref{1.8} at our disposal, to prove Theorem \ref{NeumannProblem},
we start with the same idea as that  in \cite{Kim-Shen-2008} for Laplace's equation in a convex domain.
By utilizing the real variable method developed by Shen \cite{Shen-2005, Shen-2006},
we reduce the $L^p$ estimates \eqref{1.2} for $p>2$  to the following weak reverse H\"older inequality:
 \begin{equation}\label{re-00}
        \left(\fint_{I(x_0,r)}|(\nabla u)^* +(\phi)^* |^p \right)^{1/p}
        \leq C
        \left(\fint_{I(x_0, 4r)}|(\nabla u)^* + (\phi)^* |^2 \right)^{1/2},
    \end{equation}
    where $(u, \phi)$ satisfies \eqref{local-sol}.
    We then apply the square function estimates for the Stokes equations as well as the interior estimates  to show that
\begin{equation}\label{1.9}
\aligned
           \fint_{I_r} |(\nabla u)^* +(\phi)^* |^p \leq &
            C r^{\gamma+1}\sup_{x\in D_{2r}}\bigg\{ (|\nabla
        u| + | \phi|) ^{p-2}[\delta(x)]^{1-\gamma} \bigg\}\fint_{D_{2r}}
       \left(  |\nabla^2 u| +|\nabla \phi|   \right)^2\\
&       \qquad   +C\bigg(\fint_{I_{4r}}|(\nabla u)^*|^2\bigg)^{p/2}
\endaligned
\end{equation}
for any $\gamma>0$,
where $I_r =I(x_0, r)$,  $D_r=D(x_0, r)$, and $\delta(x)=\text{\rm dist}(x, \partial\Omega)$.
The estimate \eqref{1.8} is  used to control the integral of
$(|\nabla^2 u| + |\nabla \phi| )^2$ in the right-hand side of \eqref{1.9}.
To handle the term involving $\sup_{x\in D_{2r}}$, we consider
 the  weak reverse H\"older inequality, 	
\begin{equation}\label{rev-0}
\left(\fint_{D_{3r}}( |\nabla u| +|\phi|)^q
		\right)^{1/q}  \leq C
		\left(\fint_{D_{4r}}( |\nabla u | +|\phi|)^2\right)^{\frac{1}{2}},
\end{equation}
where $(u, \phi )$ satisfies  \eqref{local-sol}.
 The estimate \eqref{rev-0} with $q=p_d=\frac{2d }{d-2}$  for $d\ge 3$ follows from \eqref{1.8} by using a Sobolev  inequality.
The self-improving property of the reverse H\"older inequality yieldst \eqref{rev-0}  for $q=p_d +\e_0$ and $d\ge 3$, where $\e_0>0$
depends on the Lipschitz character of $\Omega$.
 If $d=2$, \eqref{rev-0} holds for any $q>2$.
This, together with the interior estimates, allows us to reduce  the power of $\delta(x)$ in \eqref{1.9}, which needs to be nonnegative, to
\begin{equation*}
   -\frac{d}{q} (p-2)  + 1 -\gamma.
   \end{equation*}
By choosing $\gamma$ sufficiently small, we obtain   the upper bounds  for $p$ in Theorem \ref{NeumannProblem}.

To prove Theorem \ref{W1p},  by a duality argument, one only needs to consider the case $F=0$ and $g=0$ for $p>2$.
In this case, the results follow from \eqref{rev-0} by using the real variable method.




\section{ $W^{2,2}$ estimates}\label{W-2-2}

The  goal of this section is to establish a localized  $W^{2,2}$ estimate for the Stokes equations in a convex domain with the Neumann boundary condition.
As indicated in the Introduction,
the estimate  plays a key role in the proofs of Theorems \ref{NeumannProblem} and  \ref{W1p}.

For $x_0 \in \partial\Omega$ and $0< r< r_0$, let $D(x_0, r) = B(x_0, r)\cap \Omega $ and $I(x_0, r)=B(x_0, r) \cap \partial\Omega$.

\begin{thm} \label{W22}
    Let $\Omega$ be a bounded convex domain with $C^2$ boundary.
     Suppose that $u\in C^2(\overline{D(x_0, 2r)}; \mathbb{R}^d)$, $\phi \in C^1(\overline{D(x_0, 2r)})$, and
    \begin{equation}\label{2.0}
    \left\{
    \aligned
      -  \Delta u +\nabla \phi   & = 0 & \quad & \mbox{ in }  D(x_0, 2r), \\
        \text{\rm div}  (u)   & = 0 & \quad &   \mbox{ in }  D(x_0, 2r),\\
        \frac{\partial u}{\partial \nu} & =0 & \quad & \mbox{ on }  I(x_0, 2r),
    \endaligned
    \right.
    \end{equation}
    where $x_0\in \partial\Omega$ and $0< r< r_0$.
  Then,
    \begin{equation*}
        \int_{D(x_0,r)}( |\nabla^2 u| +|\nabla \phi|) ^2 \leq
        \frac{C}{r^2}\int_{D(x_0,2r)} ( |\nabla u| + |\phi|) ^2,
    \end{equation*}
where $C$ depends only on $d$.
\end{thm}

The  proof of Theorem  \ref{W22} relies on  a well-known  integral identity.
In the  next lemma, we use $v_{\tan}$ to denote the tangential component of a vector $v$ and
$\nabla_{\tan} w$ the tangential gradient of  a scalar $w$.
That is, $v_{\tan}= v-(v\cdot n) n$ and $\nabla_{\tan} w = (\nabla w)_{\tan}= \nabla w -(\nabla w \cdot n) n$.

\begin{lemma}\label{convex}
    Let $\Omega$ be a bounded  domain with $C^2$ boundary. Suppose that $v\in C^1(\overline{\Omega}; \R^d)$. Then
    \begin{align}
        \int_{\Omega} |\text{\rm div} (v)|^2\,dx-\sum_{i,j=1}^d\int_{\Omega}\frac{\partial v_i}{\partial x_j}\frac{\partial v_j}{\partial x_i}\,dx &=
        -2\int_{\partial\Omega}v_{\tan}\cdot\nabla_{\tan}(v\cdot n)\, d\sigma\\\nonumber
        &-\int_{\partial\Omega}\big(\beta(v_{\tan}; v_{\tan})+(tr\beta)|v\cdot n |^2\big)\, d\sigma,
    \end{align}
where $\beta(\cdot;\cdot)$ is the second fundamental quadratic form of $\partial\Omega$ and $tr\beta$ denotes the trace of $\beta$.
\end{lemma}
\begin{proof}
See \cite[pp.134-138]{Grisvard}.
\end{proof}


\begin{proof}[Proof of Theorem \ref{W22}]
Let $\xi \in C_{0}^{\infty}(\mathbb{R}^d)$ be a cut-off function such that
$$
\xi= \begin{cases}1 & \text { on } B(x_0,r) ,\\[0.1cm]
 0 & \text { outside } B(x_0, \frac{3}{2}r),
 \end{cases}
$$
and $|\nabla \xi|\leq C/r$.
For each $1\le \beta\le d$,  set
\begin{equation}\label{w}
w_\beta=\bigg(\{\partial_k u_\beta-\delta_{k\beta}\phi\}\xi\bigg)_{k=1}^d.
\end{equation}
Then $w_\beta\in C^1(\overline{\Omega}; \R^d)$.
Since $\frac{\partial u}{\partial\nu}=0$ on $I(x_0,2r)$,
we have  $ w_\beta\cdot n =0$ on $\partial\Omega$.
By a  direct computation and using the equation $\Delta u=\nabla \phi$, we obtain
\begin{align}\label{3.1}
\text{div}(w_\beta)=\big(\partial_k u_\beta-\delta_{k\beta}\phi\big)\partial_k\xi,
\end{align}
where the repeated indix $k$ is summed from $1$ to $d$.
Moreover, for $1\leq i, j \leq d$,
\begin{align}\label{3.3}
\frac{\partial (w_\beta)_i}{\partial x_j}
=\big(\partial_j\partial_i u_\beta-\delta_{i\beta}\partial_j\phi\big)\xi+\big(\partial_i u_\beta-\delta_{i\beta}\phi\big)\partial_j\xi.
\end{align}
It follows  that
\begin{align}\label{3.5}
\frac{\partial (w_\beta)_i}{\partial x_j} \cdot
\frac{\partial (w_\beta)_j}{\partial x_i}&=\big(\partial_j\partial_i u_\beta\partial_i\partial_j u_\beta+\partial_\beta \phi \partial_\beta \phi\big)\xi^2\\\nonumber
&\quad+\big(\partial_j\partial_i u_\beta\partial_j u_\beta-\delta_{i\beta}\partial_j\phi\partial_j u_\beta+\delta_{i\beta}\partial_\beta \phi \phi\big)\xi\partial_i\xi\\\nonumber
&\quad\quad+\big(\partial_i u_\beta\partial_i\partial_j u_\beta-\delta_{j\beta}\partial_i\phi\partial_i u_\beta+\delta_{j\beta}\partial_\beta \phi \phi \big)\xi\partial_j\xi\\\nonumber
&\quad\quad\quad+\big(\partial_i u_\beta\partial_j u_\beta-\delta_{j\beta}\partial_i u_\beta\phi-\delta_{i\beta}\partial_j u_\beta \phi +\delta_{i\beta}\delta_{j\beta}\phi^2\big)\partial_i\xi\partial_j\xi,
\end{align}
where the repeated indices are summed  from $1$ to $d$ and we have used the fact $\text{\rm div}(u)=0$.
We now   apply Lemma \ref{convex} with $v=w_\beta$.
Since $v\cdot n=0$ on $\partial\Omega$ and
$$
-\beta(v_{\tan}; v_{\tan})\geq 0~~~\text{on}~~~~\partial\Omega,
$$
which follows from the assumption that $\Omega$ is convex,
we obtain
$$
        \int_{\Omega} |\text{\rm div}(v)|^2\,dx\geq \sum_{i,j=1}^d\int_{\Omega}\frac{\partial v_i}{\partial x_j}\frac{\partial v_j}{\partial x_i}dx.
$$
By the Cauchy inequality, this, together with \eqref{3.1} and \eqref{3.5}, yields
\begin{equation}\label{3.6}
        \int_{D(x_0,r)} \big(|\nabla^2 u|^2+|\nabla \phi|^2\big) \leq
        \frac{C}{r^2}\int_{D(x_0,\frac{3}{2}r)} \big(|\nabla u|^2+|\phi|^2\big),
\end{equation}
where $C$ depends only on $d$.
\end{proof}

\begin{cor}\label{new-cor}
Let $(u, \phi)$ be the same as in Theorem \ref{W22}.
 Then
\begin{equation}\label{new-1}
\left(\fint_{D(x_0, r)} \left( |\nabla u| +|\phi|\right)^p \right)^{1/p}
\le C \left(\fint_{D(x_0, 2r)} \left( |\nabla u| +|\phi|\right)^2 \right)^{1/2}
\end{equation}
for  $2< p< \frac{2d}{d-2} +\e$, where $d\ge 3$ and $C>0$, $\e>0$ depend only on the Lipschitz character of $\Omega$.
If $d=2$, the estimate \eqref{new-1} holds for any $p>2$.
\end{cor}

\begin{proof}
 We  apply   the Sobolev  inequality,
 \begin{equation}\label{Sob}
 \left(\fint_{D(x_0, r)} |w|^{p_d} \right)^{1/p_d}
 \le C r \left(\fint_{D(x_0, r)} |\nabla w|^2 \right)^{1/2}
 + C \left(\fint_{D(x_0, r)} |w|^2 \right)^{1/2},
 \end{equation}
  to $(\nabla u, \phi)$,
 where $p_d>2$ if $d=2$, and $p_d=\frac{2d}{d-2}$ if $d\ge 3$.
 The constant $C$ in \eqref{Sob} depends only on  the Lipschitz character of $\Omega$.
 As a result, we obtain
 $$
 \aligned
 \left(\fint_{D(x_0, r)} ( |\nabla u|+ |\phi|) ^{p_d} \right)^{1/p_d}
  & \le C r \left(\fint_{D(x_0, r)}( |\nabla ^2 u| +|\nabla \phi| )^2 \right)^{1/2}
 + C \left(\fint_{D(x_0, r)}  (|\nabla u| + |\phi|) ^2 \right)^{1/2}\\
 &\le C \left(\fint_{D(x_0, 2r)}  (|\nabla u| + |\phi|) ^2 \right)^{1/2},
\endaligned
$$
where we have used Theorem \ref{W22}.
Note that by interior estimates for the Stokes equations (see Lemma \ref{InteriorEstimates}), the $W^{2, 2}$ estimate  in Theorem \ref{W22} continues to hold if  $x_0\in \Omega$.
It follows that the estimate \eqref{new-1}  also holds for $p=p_d$ and  $x_0\in \Omega$.
This allows us to use the self-improving property of the reverse H\"older inequality to upgrade the estimate \eqref{new-1} from $p_d$ to $p_d+\e$ for $d\ge 3$.
\end{proof}



\section{Nontangential-maximal-function estimates}

In this section we give the proof of Theorem \ref{NeumannProblem}. We first  establish a sufficient condition
 for the $L^p$ solvability of the Neumann problem \eqref{stokesSystem} for $p>2$.
 Recall that $I(x_0, r) =B(x_0, r) \cap \partial\Omega$.

\begin{thm} \label{SufficientCondition}
    Let $\Omega$ be a bounded Lipschitz domain in $\rn{d}$, $d\geq 2$, and $p>2$.  Suppose that there
    exists $C>0$ such that for any $x_0\in
    \partial\Omega$ and $0<r<r_0$, the weak reverse H\"{o}lder inequality,

    \begin{equation} \label{reverseHolder}
        \left(\fint_{I(x_0,r)}|(\nabla u)^* + (\phi)^* |^p \right)^{1/p}
        \leq C
        \left(\fint_{I(x_0,2r)}|(\nabla u)^* + (\phi)^* |^2 \right)^{1/2},
    \end{equation}

    \noindent holds for any solution $(u,\phi ) $ of
    \begin{equation}
		-\Delta u+\nabla \phi=0  \quad \text{in}\quad \Omega ~~~~\text{and} \\
		~~~~\text{\rm div}  (u) = 0\quad \text{in}\quad \Omega,
	\end{equation}
     with the properties that
    $(\nabla u)^*+ (\phi)^* \in L^2(\partial\Omega)$ and
    $\frac{\partial u}{\partial\nu}=0$ on $I(x_0,3r)$. Then the $L^p$ Neumann problem for the Stokes equations
    (\ref{stokesSystem}) in $\Omega$ is uniquely solvable.
\end{thm}

The proof of Theorem \ref{SufficientCondition} relies on the following real variable argument proved by
Shen in \cite{Shen-2007}.

\begin{thm}\label{realVariableArgument}
    Let $Q_0$ be a surface cube on $\partial\Omega$ and $F\in L^2(2Q_0)$.  Let
    $p>2$ and $f\in L^q(2Q_0)$ for some $2<q<p$.  Suppose that for
    each dyadic subcube $Q$ of $Q_0$ with $|Q|\leq \beta |Q_0|$,
    there exists two integrable functions $F_Q$ and $R_Q$ on $2Q$
    such that
    \begin{eqnarray*}
        |F| &\leq & |F_Q|+|R_Q|~~ on ~~2Q, \\
        \left( \fint_{2Q} |R_Q|^p\,d\sigma
        \right)^{1/p} &\leq & C_1\left\{\left( \fint_{\alpha
        Q} |F|^2\,d\sigma\right)^{1/2} + \sup_{Q'\supset Q} \left(
        \fint_{Q'} |f|^2\,d\sigma\right)^{1/2}\right\}, \\
        \left(\fint_{2Q} |F_Q|^2\,d\sigma\right)^{1/2}
        &\leq &
        C_2\sup_{Q'\supset Q} \left(\fint_{Q'}
        |f|^2\,d\sigma\right)^{1/2},
    \end{eqnarray*}

    \noindent where $C_1,C_2>0$ and $0<\beta<1<\alpha$.  Then,

    \begin{eqnarray}
        \left( \fint_{Q_0} |F|^q\,d\sigma\right)^{1/q}
        \leq  C\left\{\left(\fint_{2Q_0} |F|^2\,d\sigma\right)^{1/2} +
        \left(\fint_{2Q_0}
        |f|^q\,d\sigma\right)^{1/q} \right\}, \label{shenThmConclusion}
    \end{eqnarray}

    \noindent where $C$ depends only on $p$, $q$, $C_1$,
    $C_2$, $\alpha$, $\beta$, and $d$.
\end{thm}

\begin{proof}[Proof of Theorem \ref{SufficientCondition}\mbox{}]



        Let $p>2$.
        For $g\in L_0^p(\partial\Omega; \R^d)$, let $(u,\phi)$ be the solution
        of the $L^2$ Neumann problem with data $g$.
        That is, $(u, \phi)$ satisfies the Stokes equations in $\Omega$,
        $\frac{\partial u}{\partial \nu} =g$ on $\partial\Omega$, and
        \begin{equation}\label{est-2}
        \| (\nabla u)^* \|_{L^2(\partial\Omega)}
        + \| (\phi)^* \|_{L^2(\partial\Omega)}
        \le C \| g\|_{L^2(\partial\Omega)}.
        \end{equation}
        We will use Theorem \ref{realVariableArgument} to  show that
        \begin{equation}\label{est-q}
        \| (\nabla u)^* \|_{L^q(\partial\Omega)}
        + \| (\phi)^* \|_{L^q(\partial\Omega)}
        \le C \| g\|_{L^q(\partial\Omega)}
        \end{equation}
        for $2<q\le p$.

      To this end,  we fix $x_0\in \partial\Omega$ and $0< r< r_0$.
      Choose a cut-off function $\varphi \in C_{0}^{\infty}(\R^d)$ such that $0\le \varphi \le 1$,
     $$
     \varphi= \begin{cases}1, & \text { in } B(x_0, 8r)  ,\\
      0, & \text { outside of  } B(x_0, 10r) .\end{cases}
     $$
     Decompose
        $u=u_1+u_2, \phi=\phi_1+\phi_2$,  where $(u_\ell ,\phi_\ell )$ is the solution of the $L^2$ Neumann problem,
\begin{equation}\label{u1}
    \left\{
    \aligned
        -\Delta u_\ell  + \nabla \phi_\ell & =0  & \quad & \mbox{ in } \Omega, \\
        \diverg{(u_\ell) }  & = 0  & \quad &  \mbox{ in }\Omega, \\
        \frac{\partial u_\ell}{\partial \nu} & =g_\ell  & \quad & \mbox{ on }
        \partial\Omega,
    \endaligned
    \right.
    \end{equation}
for $\ell =1, 2$, and
\begin{equation}\label{u2}
    \left\{
    \aligned
 g_1  & = (g +\alpha) \varphi -\alpha   \\
 g_2  & = (g+\alpha) (1-\varphi).
 \endaligned
 \right.
     \end{equation}
In \eqref{u2}, $\alpha\in \R^d$ is a constant so that $\int_{\partial\Omega} g_1=0$, i.e.,
$$
\alpha = \int_{\partial\Omega} g \varphi \Big/ \int_{\partial\Omega} (1-\varphi).
$$

Next, for $Q=I(x_0, r)$, let
        \begin{eqnarray*}
            F = (\nabla u)^* + (\phi)^*,~~
            f = |g|, ~~
            F_Q = (\nabla u_1)^*+ (\phi_1)^*, ~~
            R_Q = (\nabla u_2)^* + (\phi_2)^*.
        \end{eqnarray*}
By the $L^2$ estimates for the Neumann problem(see \cite{FKV-1988}),  we obtain
        {\allowdisplaybreaks
        \begin{eqnarray*}
            \fint_{2Q} |F_Q|^2  &\leq&
            \frac{1}{|2Q|} \int_{\partial\Omega}
            |F_Q|^2  \\
            &\leq& \frac{C}{|Q|} \int_{\partial\Omega}
            |g_1|^2 \\
            &\leq&  C \fint_{10Q }
            |g|^2
            \leq C\sup_{Q'\supset Q} \frac{1}{|Q'|} \int_{Q'}
            |f|^2,
        \end{eqnarray*}
        }
        where we have used the observation
        $$
        |\alpha | \le C \int_{10 Q} |g|.
        $$

        Note that
        $\frac{\partial u_2}{\partial\nu}=0$ on $I(x_0, 8r) $.  As a result, by applying
        the weak reverse H\"{o}lder inequality
        (\ref{reverseHolder}),  we obtain
        \begin{equation} \label{RQEst}
        \aligned
            \left( \fint_{2Q}
            |R_Q|^{p}\right)^{1/p}
            &\leq C\left(\fint_{4Q}
            |R_Q |^2 \right)^{1/2} \\
            &\leq C \left( \fint_{4Q}
            |F|^2 \right)^{1/2} +
            C\left(\fint_{4Q}
            |F_Q|^2  \right)^{1/2}\\
            &\leq C\left(\fint_{4Q} |F|^2\right)^{1/2}  +
            C\sup_{Q'\supset Q} \left(\fint_{Q'}
            |f|^2 \right)^{1/2}.
        \endaligned
        \end{equation}
Thus, the conditions of Theorem \ref{realVariableArgument} are
        satisfied. Hence,  for $2<q< p$ and $y_0 \in \partial\Omega$,
        \begin{eqnarray}
            \left(\fint_{I(y_0,cr_0)}
            |(\nabla u)^* + (\phi )^*|^q \right)^{1/q} &\leq&
            C\left(\fint_{I(y_0,r_0)}
            |(\nabla u)^* +(\phi)^* |^2 \right)^{1/2} \nonumber \\ && \qquad+
            C\left(\fint_{I(y_0, r_0) } |g|^q \right)^{1/q}.
            \label{ballReverse}
        \end{eqnarray}
        Finally, we cover $\partial\Omega$ with a finite number of
        balls of radius $cr_0$. It follows  that
           $$
           \aligned
           \| (\nabla u)^* \|_{L^q(\partial\Omega)} + \| (\phi)^* \|_{L^q(\partial\Omega)}
           &  \le C \left\{
             \| (\nabla u)^* \|_{L^2(\partial\Omega)} + \| (\phi)^* \|_{L^2(\partial\Omega)}
             + \| g \|_{L^q(\partial\Omega)} \right\}\\
             &\le C \| g \|_{L^q(\partial\Omega)},
             \endaligned
             $$
             where we have used the $L^2$ estimate \eqref{est-2} as well as H\"older's inequality.
             Note that  the reverse H\"older condition \eqref{reverseHolder} is self-improving;
             the condition \eqref{reverseHolder} for $p$ implies \eqref{reverseHolder} for some $\bar{p}>p$.
             As a result,  the $L^q$ estimate \eqref{est-q} holds for $2< q< \bar{p}$.
             In particular, it holds for $q=p$.
           \end{proof}

\begin{lemma}\label{InteriorEstimates}
    Let $(u, \phi) $ be a solution of the Stokes equations,
    $-\Delta u +\nabla \phi =0$ and $\text{\rm div}(u)=0$ in $B(x_0, 2r)$.  Then
    $$
   r \left(  |\nabla^2  u(x_0)| + |\nabla \phi (x_0)|  \right)
   + \left( |\nabla u (x_0)| +|\phi (x_0)|\right) \leq
    C \fint_{B(x_0,r)} \left( |\nabla u| + |\phi|\right) ,
    $$
    where $C$ depends only on $d$.
\end{lemma}
\begin{proof}
This follows from  the well-known interior estimates for the Stokes equations.
\end{proof}

Recall that the square function $S(w)$ is defined as
\begin{equation}
    S(w)(z) = \left(\int_{\Gamma(z)} \frac{|\nabla
    w(x)|^2}{|x-z|^{d-2}}\,dx\right)^{1/2}
\end{equation}
for $z\in \partial\Omega$, where $\Gamma(z)=\{ x\in \Omega: |x-z|< C_0 \text{\rm dist} (x, \partial\Omega) \}$.

\begin{lemma}\label{squarefunction}
Let $\Omega$ be a bounded Lipschitz domain.
Assume that  $-\Delta u + \nabla \phi=0 $ and $\text{\rm div}(u)=0$ in $\Omega$.
Then, for $0< p< \infty$,
\begin{equation}
      \left\| ( \nabla u)^* +(\phi)^* \right\|_{L^p(\partial\Omega)}
    \leq C \left\|
    S(\nabla u) + S(\phi) \right\|_{L^p(\partial\Omega)} +
    C\left\{ |\nabla u(y_0)| + |\phi (y_0) | \right\}  |\partial\Omega|^{1/p}, \label{sqFunct2}
\end{equation}
where $y_0 \in \Omega$, $\delta (y_0) \ge c_0 \, \text{\rm diam}(\Omega)$, and
 $C$ depends on $p$ and
the Lipschitz character of $\Omega$.
\end{lemma}

\begin{proof}
Since $\partial_j u$ is also a solution of the Stokes equations, it follows from  \cite{BS-1995} that
\begin{equation}
      \left\| ( \nabla u)^* \right\|_{L^p(\partial\Omega)}
    \leq C \left\|
    S(\nabla u) \right\|_{L^p(\partial\Omega)} +
    C |\nabla u(y_0)|    |\partial\Omega|^{1/p}. \label{sqFunct3}
\end{equation}
Also note that $\phi$ is harmonic in $\Omega$.
Thus, by \cite{Dahlberg-1980},
\begin{equation}
      \left\| (\phi)^* \right\|_{L^p(\partial\Omega)}
    \leq C \left\|
    S(\phi) \right\|_{L^p(\partial\Omega)} +
    C  |\phi (y_0)  |\partial\Omega|^{1/p}. \label{sqFunct4}
\end{equation}
The estimate \eqref{sqFunct2} follows from \eqref{sqFunct3}-\eqref{sqFunct4}.
\end{proof}

\begin{lemma} \label{importantLemma}
    Let $(u, \phi)$ be the same as in Lemma \ref{squarefunction}.
     Let $p>2$. Then for any $\gamma \in
    (0,1)$, we have
    \begin{equation}\label{sq-est}
    \aligned
       &   \int_{\partial\Omega} |(\nabla u)^* +(\phi)^* |^p\,d\sigma
        \leq C
        |\partial\Omega| \left\{ |\nabla u(y_0)|^{p}  + |\phi (y_0)|^p \right\}\\
     &    + C_{\gamma} \{\mbox{diam}(\Omega)\}^{\gamma}
        \sup_{x\in \Omega} \bigg\{ \left( |\nabla^2 u(x) | + |\nabla \phi (x) |\right)
        ^{p-2}[\delta(x)]^{p-1-\gamma}\bigg\} \int_{ \Omega}
      \left(   |\nabla^2u| +|\nabla \phi| \right) ^2 dy.
        \endaligned
    \end{equation}
\end{lemma}

\begin{proof}

We first rewrite $S( w)(z)$ as
    $$
    S(w)(z) = \left(\int_{\Gamma(z)} \frac{|\nabla
    w(x)|^2}{|x-z|^{\frac{2(d+\gamma-p)}{p}}}\cdot
    \frac{dx}{|x-z|^{\frac{(p-2)d-2\gamma}{p}}}\right)^{1/2}.
    $$
    It follows by H\"{o}lder's inequality that
    \begin{eqnarray*}
        S( w)(z) &\leq& \left(\int_{\Gamma(z)}
        \frac{|\nabla
        w(x)|^p}{|x-z|^{d+\gamma-p}}\,dx\right)^{1/p}\left(\int_{\Gamma(z)}
        \frac{dx}{|x-z|^{d-\frac{2\gamma}{p-2}}}\right)^{\frac{p-2}{2p}}
         \\
        &\leq&
        C_\gamma \{\mbox{diam}(\Omega)\}^{\gamma/p}\left(\int_{\Gamma(z)}
        \frac{|\nabla
        w(x)|^p}{|x-z|^{d+\gamma-p}}\,dx\right)^{1/p}.
    \end{eqnarray*}

    \noindent Next, by integrating $|S( w)(z)|^p$ over
    $\partial\Omega$,  we have
    \begin{equation}\label{3.8}
        \int_{\partial\Omega} |S( w)|^p\,d\sigma \leq
        C_{\gamma}\{\mbox{diam}(\Omega)\}^{\gamma } \int_{\Omega}
        |\nabla w(x)|^p[\delta(x)]^{p-1-\gamma}\,dx.
    \end{equation}
    Thus,
    \begin{equation}
    \aligned
&        \int_{\partial\Omega} |S( \nabla u) +S(\phi) |^p\,d\sigma \leq
        C_{\gamma}\{\mbox{diam}(\Omega)\}^{\gamma } \int_{\Omega}
     \left(    |\nabla^2 u(x)| + |\nabla \phi (x) | \right) ^p[\delta(x)]^{p-1-\gamma}\,dx\\
     &  \le C_{\gamma} \{\mbox{diam}(\Omega)\}^{\gamma}
        \sup_{x\in \Omega} \bigg\{ \left( |\nabla^2 u(x) | + |\nabla \phi (x) |\right)
        ^{p-2}[\delta(x)]^{p-1-\gamma}\bigg\} \int_{ \Omega}
      \left(   |\nabla^2u| +|\nabla \phi| \right) ^2 dy.
\endaligned
    \end{equation}
This, together with \eqref{sqFunct2}, gives \eqref{sq-est}.
\end{proof}

    We are ready to give the proof of Theorem \ref{NeumannProblem}.

    \begin{proof}[Proof of Theorem \ref{NeumannProblem}]

        Let $0<r< r_0$ and $x_0\in \partial\Omega$.
        Let $(u,\phi)$ be a solution of $-\Delta u + \nabla \phi=0$, $\mbox{div}(u)=0$ in $\Omega$.
         Suppose that  $\frac{\partial u}{\partial \nu}=0$ on $I(x_0, 8r)$
         and  $(\nabla u)^* + (\phi)^* \in L^2(\partial\Omega)$.
          By Theorem \ref{SufficientCondition}, it suffices to show that for $p$ satisfying \eqref{p},
        \begin{equation}\label{WRH}
        \left(\fint_{I(x_0, r) }|(\nabla u)^*+ (\phi)^* |^p\right)^{1/p}
        \leq C
        \left(\fint_{I(x_0, 2r) }|(\nabla u)^* +(\phi)^* |^2\right)^{1/2}.
         \end{equation}

To this end, let
	\begin{equation*}
		\mathcal{M}_r( w)(z)=\sup\{| w(x)|: x\in\Gamma(z)\quad \text{and}\quad \delta(x)<  c_0 r\},
	\end{equation*}
and
	\begin{equation*}
		\mathcal{M}^r( w)(z)=\sup\{| w(x)|: x\in\Gamma(z)\quad \text{and}\quad \delta(x)\geq  c_0 r\},
	\end{equation*}
	for $z\in \partial\Omega$, where $\delta (x) =\text{\rm dist}(x, \partial\Omega)$.
	Using the interior estimates for $(u, \phi)$, it is not difficult to see that
	
	 \begin{equation}\label{WRH-1}
        \left(\fint_{I(x_0, r) }| \mathcal{M}^r (|\nabla u| +|\phi| )  |^p\right)^{1/p}
        \leq C
        \left(\fint_{I(x_0, 2r) }|(\nabla u)^* +(\phi)^* |^2\right)^{1/2}.
        \end{equation}
To estimate $\mathcal{M}_r(\nabla u)$, notice that
        \begin{equation}\label{local-1}
            \int_{I(x_0, r) } |\mathcal{M}_r(|\nabla u| +|\phi| )|^p\,d\sigma \leq \int_{\partial D(x_0, 2r) } |(\nabla u)^* +(\phi)^* |^{p}\,d\sigma,
        \end{equation}
where the nontangential maximal functions in the right-hand side of \eqref{local-1} are taken with respect to the domain $D(x_0, 2r)$.
Using Lemma \ref{importantLemma}, we obtain
 \begin{equation}\label{local-2}
 \aligned
          &   \int_{I(x_0, r) } |\mathcal{M}_r(| \nabla u|+ |\phi| )|^p\,d\sigma
           \leq  C r^{d-1} \left(  |\nabla u (y_0) | + |\phi (y_0)| \right) ^{p} \\
        & + C_{\gamma} r^{\gamma}
        \sup_{x\in D(x_0, 2r) }\left(  |\nabla^2
        u| + |\nabla \phi| \right) ^{p-2}[\delta(x)]^{p-1-\gamma} \int_{D(x_0, 2r) }
       \left(  |\nabla^2u  | +|\nabla \phi  |\right)^2 dy,
       \endaligned
        \end{equation}
        \noindent where $y_0\in D(x_0, r)$ is chosen so that $\delta(y_0)\approx r$.
        By the interior estimates,
        \begin{equation}\label{4.1}
        \aligned
        |\nabla u(y_0) | +|\phi (y_0)|
       &  \le
       \left(\fint_{B(y_0,c_0 r)} ( |\nabla u| +|\phi|) ^2\right)^{1/2}\\
       & \leq
            C \left(\fint_{I (x_0, 3r)}  |(\nabla u)^* + (\phi)^* |^2\right)^{1/2}.
        \endaligned
        \end{equation}
        It follows from \eqref{WRH-1}, \eqref{local-2} and  \eqref{4.1} that
        \begin{equation}\label{4.2}
        \aligned
&            \fint_{I(x_0, r)} |(\nabla u)^* + (\phi)^* |^p \leq C\bigg(\fint_{I(x_0, 3r)} |(\nabla u)^*+(\phi)^* |^2\bigg)^{p/2}\\
   &
           +  C r^{\gamma +1}\sup_{x\in D(x_0, 2r) }  (  |\nabla^2  u | + |\nabla \phi | )^{p-2}[\delta(x)]^{p-1-\gamma} \fint_{D(x_0, 2r)}
       (  |\nabla ^2 u| + |\nabla \phi|) ^2.
        \endaligned
        \end{equation}

        To handle the second term in the right-hand side of \eqref{4.2}, we utilize  the assumptions that $\Omega$ is convex and  $\frac{\partial u}{\partial \nu}=0$ on
        $I(x_0, 8r)$, which allow us to apply Theorem \ref{W22} and Corollary \ref{new-cor}.
         For any $x\in D(x_0, 2r)$,
        \begin{equation}\label{4.3}
        \aligned
           |\nabla^2 u(x)| +|\nabla \phi (x)|  & \leq \frac{C}{\delta(x)} \left(\fint_{B(x, \delta(x)/2) } (  |\nabla u| +|\phi|) ^{\bar{p}}  \right)^{1/\bar{p}}\\
 &\leq \frac{Cr^{\frac{d}{\bar{p}}  } } {[\delta(x)]^{1+ \frac{d}{\bar{p}}}}
 \left(\fint_{D(x_0, 3r) } ( |\nabla u| +|\phi |) ^{\bar{p}} \right)^{1/\bar{p}}\\
 &\leq \frac{Cr^{\frac{d}{\bar{p}}  } } {[\delta(x)]^{1+ \frac{d}{\bar{p}}}}
 \left(\fint_{D(x_0, 6r) } ( |\nabla u| +|\phi |) ^{2} \right)^{1/2},
    \endaligned
        \end{equation}
        where $\bar{p}>2$ for $d=2$, and $p_d< \bar{p}< p_d +\e$ for $d\ge 3$.
        Plugging  \eqref{4.3} into \eqref{4.2} and using Theorem \ref{W22},  we obtain
        \begin{equation}\label{4.4}
        \aligned
&           \fint_{I(x_0, r)} |(\nabla u)^* + (\phi)^* |^p \leq C\bigg(\fint_{I(x_0, 3r)} |(\nabla u)^*+(\phi)^* |^2\bigg)^{p/2}\\
 &\qquad\qquad  +     C r^{\frac{d}{\bar{p}} (p-2)  +\gamma-1} \sup_{x\in D(x_0, 2r) } [\delta(x)]^{(1+\frac{d}{\bar{p}}) (2-p) +p -1-\gamma} \left(\fint_{D(x_0, 6r) }
       ( |\nabla u| +|\phi| )^2\right)^{p/2}.
        \endaligned
        \end{equation}
Note that if
\begin{equation}\label{new-3}
{(1+\frac{d}{\bar{p}}) (2-p) +p -1}
=1+ \frac{d}{\bar{p}} (2-p) >  0,
\end{equation}
we may choose $\gamma>0$ so small that $ 1+ \frac{d}{\bar{p}} (2-p)-\gamma >0$. It follows that
        $$
        \aligned
           \fint_{I(x_0, r)} |(\nabla u)^* +(\phi)^* |^p &\leq
            C \bigg(\fint_{I(x_0, 3r)}  |(\nabla u)^* +(\phi)^* |^2\bigg)^{p/2}+C\bigg(\fint_{I(x_0, 6r)} |(\nabla u)^* +(\phi)^* |^2\bigg)^{p/2}\\\nonumber
            &\leq C\bigg(\fint_{I(x_0, 6r)} |(\nabla u)^*+(\phi)^* |^2\bigg)^{p/2}.
        \endaligned
        $$
 By covering $I(x_0, r)$ with a finite number of $I(z_j, r/4)$ with $z_j \in I(x_0, r)$, we obtain \eqref{WRH}.

 Finally, observe that  \eqref{new-3} is equivalent to $p< 2+ \frac{\bar{p}}{d}$.
 Since $\bar{p}>\frac{2d}{d-2}$ for $d\ge 3$ and $\bar{p}>2$ is arbitrary for $d=2$,
 we conclude that \eqref{WRH}  holds for $2<p< \frac{2d}{d-2}+\e$  if $d\ge 3$, and for any $2<p< \infty$ if $d=2$.
This completes the proof.
   \end{proof}

%

          \begin{remark}
          The proof of Theorem  \ref{NeumannProblem} uses  the localized $W^{2, 2}$ estimate in
          Theorem \ref{W22}, However, in Theorem \ref{W22}, we assume that $\Omega$ is a smooth convex domain and $u$  the
          smooth solution. To resolve this issue, we use an approximate argument. More precisely,
           we first approximate a general convex domain $\Omega$ by  a sequence of smooth convex domains $\{\Omega_\ell \}$
          with uniform Lipschitz characters. On each $\Omega_\ell$, we use Theorem  \ref{W22} to establish the estimate \eqref{1.2}
          for  a smooth data $g$ with bounding constant $C$  independent of $\ell$.
          Now,  for $g\in C^\infty(\mathbb{R}^d; \mathbb{R}^d)$ with $\int_{\partial\Omega} g=0$,
          let $(u^\ell, \phi^\ell)$ be the smooth solution of \eqref{stokesSystem} in $\Omega_\ell$ with Neumann data $g_\ell = g|_{\partial\Omega_\ell} - \fint_{\partial\Omega_\ell} g$.
          Let $(u, \phi)$ be the solution of \eqref{stokesSystem} in $\Omega$ with data $g|_{\partial\Omega} $ and $p=2$.
          Using the estimate \eqref{1.2}  for  $u-u_\ell$ in $\Omega_\ell$ with   $ p=2$,
           one may show that $\nabla u_\ell  \to \nabla u$ and $\phi_\ell \to  \phi $ uniformly on any compact subset of $\Omega$. This, together with the $L^p$
           estimate  for $(u_\ell, \phi_\ell)$ in
           $\Omega_\ell$, yields the $L^p$ estimate \eqref{1.2} for $(u, \phi)$ in $\Omega$. Finally, a density argument is used to establish \eqref{1.2} for a general  $g$  in $L^p_0(\partial\Omega; \mathbb{R}^d)$. We point out that a similar approximation argument is needed  in the proof of Theorem \ref{W1p}, given in the next section. We leave the details to the reader.
          \end{remark}



\section{\bf \texorpdfstring{$W^{1,p}$}{W1p} estimates}

Let $f \in L^2(\Omega; \mathbb{R}^{d\times d})$, $F\in L^2(\Omega; \mathbb{R}^d)$, and $g\in B^{-1/2, 2}(\partial\Omega; \mathbb{R}^d)$, where $\Omega$ is a bounded
Lipschitz domain in $\mathbb{R}^d$.
Suppose $(f, F, g)$ satisfies the compatibility condition \eqref{comp} with $p=2$.
To establish the existence and uniqueness of weak solutions of \eqref{inhomogeneousStokes} for $p=2$, consider the bilinear form
$$
B[u, v]=\int_\Omega \nabla u \cdot \nabla v
$$
on the Hilbert space,
$$
V= \Big\{ u \in W^{1, 2}(\Omega; \mathbb{R}^d): \ \text{\rm div}(u) =0 \text{ in } \Omega \text{ and } \int_\Omega u=0 \Big\}.
$$
By the Lax-Milgram theorem,  it follows that there exists a unique $u\in V$ such that
\begin{equation}\label{weak-11}
\int_\Omega \nabla u \cdot \nabla v
=-\int_\Omega f \cdot \nabla v + \int_\Omega F \cdot v
+ \langle g, v\rangle_{B^{-1/2, 2}(\partial\Omega) \times B^{1/2, 2} (\partial\Omega)}
\end{equation}
for any $v\in V$.
Moreover,  we have
\begin{equation}\label{energy-11}
\|\nabla u \|_{L^2(\Omega)}
\le C \left\{ \| f \|_{L^2(\Omega)} + \| F \|_{L^2(\Omega)} + \| g \|_{B^{-1/2, 2}(\partial\Omega)}\right\}.
\end{equation}

Observe that if $(f, F, g)$ are smooth functions in $\mathbb{R}^d$ and $u$ is given above,
there exists $\phi \in L^2(\Omega)$ such that
\begin{equation}\label{wf}
\int_\Omega \nabla u \cdot \nabla \varphi
-\int_\Omega \phi\,  \text{\rm div}(\varphi)
=-\int_\Omega f \cdot \nabla \varphi
+\int_\Omega F \cdot \varphi
+ \langle g, \varphi\rangle_{B^{-1/2, 2}(\partial\Omega) \times B^{1/2, 2} (\partial \Omega)}
\end{equation}
for any $\varphi \in C^\infty(\mathbb{R}^d; \mathbb{R}^d)$.
This follows from the potential theory for the Stokes equations in \cite{FKV-1988}.
By choosing $\varphi$ such that $\varphi  (x) =x-x_0$ in $\Omega$, where $x_0\in \Omega$,
it follows from   \eqref{wf} that
\begin{equation}\label{wf-1}
\Big |\int_\Omega \phi \Big|
\le C
 \left\{  \| \nabla u \|_{L^p(\Omega)} + \| f \|_{L^p(\Omega)} + \| F \|_{L^p(\Omega)} + \| g \|_{B^{-1/p, p}(\partial\Omega)}\right\}
\end{equation}
for any $1< p< \infty$.
Note that \eqref{wf} also gives
$$
-\Delta u +\nabla \phi =\text{\rm div}(f) +F \quad \text{ in } \Omega,
$$
in the sense of distributions. Thus,
\begin{equation}\label{wf-2}
\aligned
\| \phi -\fint_\Omega \phi \|_{L^p(\Omega)}
 & \le C \| \nabla \phi \|_{W^{-1, p} (\Omega)}\\
& \le C  \left\{  \|\nabla u \|_{L^p(\Omega)} + \| f \|_{L^p(\Omega)} + \| F \|_{L^p(\Omega)} \right\},
\endaligned
\end{equation}
which, together with \eqref{wf-1},  leads to
\begin{equation}\label{wf-3}
\| \phi \|_{L^p(\Omega)}
\le C  \left\{  \|\nabla u \|_{L^p(\Omega)} +  \| f \|_{L^p(\Omega)} + \| F \|_{L^p(\Omega)} + \| g \|_{B^{-1/p, p}(\partial\Omega)}\right\},
\end{equation}
for $1< p< \infty$.
In particular,  if $p=2$, we obtain the energy estimate,
\begin{equation}\label{wf-4}
\|\nabla u \|_{L^2(\Omega)} + \| \phi \|_{L^2(\Omega)}
\le C  \left\{  \| f \|_{L^2(\Omega)} + \| F \|_{L^2(\Omega)} + \| g \|_{B^{-1/2,2}(\partial\Omega)}\right\}.
\end{equation}
By a density argument we  deduce that  both \eqref{wf} and \eqref{wf-4}
hold for any $f \in L^2(\Omega; \mathbb{R}^{d\times d})$, $F\in L^2(\Omega; \mathbb{R}^d)$, and $g\in B^{-1/2, 2}(\partial\Omega; \mathbb{R}^d)$,
satisfying the condition \eqref{comp}.
We will call $(u, \phi)$ the weak solution of \eqref{inhomogeneousStokes}.

To establish the $W^{1, p}$ estimates in Theorem \ref{W1p}, we consider three cases.

\subsection{The case  $F=0$ and $ g=0$}

In this subsection we consider the Neumann problem,
   \begin{equation}\label{N-4-1}
\left\{
\aligned
        -\Delta u +\nabla \phi  & =\text{\rm div}(f)  & \quad &  \mbox{ in } \Omega, \\
        \text{\rm div}  (u) & = 0   & \quad &  \mbox{ in }\Omega, \\
        \frac{\partial u}{\partial \nu} & =-n\cdot f  & \quad &  \mbox{ on }
        \partial\Omega.
    \endaligned
    \right.
    \end{equation}
We begin with  a variant of  Theorem \ref{realVariableArgument}.

\begin{thm}\label{real2}
Let $\Omega$ be a bounded Lipschitz domain in $\mathbb{R}^d$
and $H \in L^2(\Omega)$. Let $p>2$ and $h\in L^q(\Omega)$ for some $2<q<p$.
 Suppose that for each ball $B$ centered on $\overline{\Omega}$
 with $|B|\leqslant \beta |\Omega|$, there exist two measurable functions $F_B$, $R_B$ on $\Omega\cap 2B$ such that
\begin{equation*}
  |H|\leq |H_B|+|R_B| ~~on ~~\Omega\cap 2B,
\end{equation*}
\begin{equation}\label{rvm2-1}
  \left\{\fint_{\Omega\cap 2B}|R_B|^p\right\}^{\frac{1}{p}}\leqslant C_1\left\{\left(\fint_{\Omega\cap 4B }|H|^2\right)^{\frac{1}{2}}
  +\sup_{B\subset B^{\prime}}\left(\fint_{B^{\prime}\cap \Omega}|h|^2\right)^{\frac{1}{2}}\right\},
\end{equation}
and
\begin{equation}\label{rvm2-2}
\fint_{\Omega\cap 2B}|H_B|^2dx
\leqslant C_2\sup_{B\subset
B^{\prime}}\fint_{\Omega\cap B^{\prime}}|h|^2,
\end{equation}
where $C_1, C_2>1$ and $0<\beta<1$.
 Then we have
\begin{equation}\label{real result}
\left\{\fint_{\Omega}|H|^q \right\}^{\frac{1}{q}}\leqslant
C\left\{\left(\fint_{\Omega}|H|^2\right)^{\frac{1}{2}}
+\left(\fint_{\Omega}|h|^q \right)^{\frac{1}{q}}\right\},
\end{equation}
where $C>0$ depends only on $C_1,C_2,d,p,q,\beta$ and the Lipschitz character of $\Omega$.
\end{thm}
\begin{proof}
  See \cite[Theorem 4.2.6]{Shen-book}.
\end{proof}

By using Theorem \ref{real2}, the $W^{1,p}$
estimates for Stokes systems subjected to the Neumann boundary condition   are reduced to
 a weak reverse H\"older inequality  for $p>2$.

\begin{lemma}\label{1.5}
	Let $\Omega$ be a bounded Lipschitz domain in $\mathbb{R}^d,d\geq 2$ and $p>2$.
	Suppose that
	the weak reverse H\"older inequality,
	\begin{equation}\label{1.6}
		\left(\fint_{D(x_0,r)}( |\nabla v| + |\pi|) ^p
		\right)^{\frac{1}{p}} \leq C_0
		\left(\fint_{D(x_0,2r)}( |\nabla
		v| + |\pi|)^2\right)^{\frac{1}{2}},
	\end{equation}
	holds, whenever $(v, \pi)$ is a  weak solution to
\begin{equation}\label{1.6-1}
\left\{
\aligned
			-\Delta v + \nabla \pi & =0 & \quad & \text{ in } D(x_0, 2r),\\
		\text{\rm div} (v)  & =0 &  & \text{ in } D(x_0, 2r), \\
		\frac{\partial v}{\partial \nu} & =0 & \quad & \text{ on } I (x_0, 2r),
		\endaligned
		\right.
		\end{equation}
	where  $x_0\in \partial \Omega $ and $0< r< r_0$.
		Let $(u,\phi)\in W^{1, 2} (\Omega; \mathbb{R}^d)\times L^2(\Omega)$ be a weak solution of \eqref{N-4-1}  with
		$f\in L^p(\Omega; \mathbb{R}^{d\times d})$.
	Then $u\in W^{1,p}(\Omega; \mathbb{R}^d )$, $\phi \in L^p(\Omega)$, and
	\begin{equation}\label{1.7}
		\|\nabla u\|_{L^p(\Omega)} + \| \phi \|_{L^p(\Omega)}\le C\| f\|_{L^p(\Omega)},
			\end{equation}
	where $C>0$ depends only on $d, p, C_0$ and $\Omega$.
\end{lemma}

\begin{proof}

	Let $(u,\phi)$ be the unique weak  solution to \eqref{N-4-1}
in $W^{1,2}(\Omega; \mathbb{R}^d )\times L^2(\Omega)$ with $f\in L^p(\Omega; \mathbb{R}^{d\times d})$.
 To prove \eqref{1.7}, we apply Theorem \ref{real2} with
$$
H= |\nabla u| + |\phi| \quad \text{ and } \quad  h= |f|.
$$
We verify the conditions \eqref{rvm2-1}-\eqref{rvm2-2} for the case $x_0\in \partial\Omega$.
The interior case, where $B(x_0, 2r)\subset \Omega$,  is similar.

Let $\varphi \in C_{0}^{\infty}(B(x_0,8r))$ be a cut-off function such that
$$
\varphi= \begin{cases}1, & \text { on } B(x_0,4r) ,\\[0.1cm]
 0, & \text { outside } B(x_0,8r).\end{cases}
$$
We decompose $u=v+w$ and $\phi=\pi+\theta$, where $(v,\pi)$ and $(w, \theta)$ satisfy
\begin{equation}\label{Equation v}
  \left\{
  \aligned
-\Delta v+\nabla \pi & = \text{\rm div} (\varphi f)   & \quad  & \text { in } \Omega,\\
\text{\rm div}  (v) & =0 & \quad & \text { in } \Omega, \\
\frac{\partial v}{\partial \nu} & =-\varphi f \cdot n & \quad & \text { on } \partial \Omega,
\endaligned
\right.
\end{equation}
and
\begin{equation}\label{Equation w}
  \left\{
  \aligned
-\Delta w+\nabla \theta & = \text{\rm div} ((1-\varphi ) f) & \quad  & \text { in } \Omega,\\
\text{\rm div}  (w) & =0 & \quad & \text { in } \Omega, \\
\frac{\partial w}{\partial \nu} & = (\varphi-1)  f \cdot n & \quad & \text { on } \partial \Omega.
\endaligned
\right.
\end{equation}
Let $ H_B=|\nabla v| +|\pi | $, $R_B=|\nabla w| +|\theta|$. By using the energy estimate \eqref{wf-4} for $(v, \pi )$,  we obtain that
$$
\begin{aligned}
\fint_{2B\cap\Omega}\left|H_{B}\right|^2 &
\le\frac{C}{|2B\cap\Omega|}\int_{\Omega}( |\nabla v| + |\pi|) ^{2}  \\&\leq  \frac{C}{|2B\cap\Omega|}\int_{\Omega}  |f\varphi| ^{2}  \\
&\le C\fint_{8B\cap\Omega}|f|^{2} .
\end{aligned}
$$
This gives \eqref{rvm2-2}. To see \eqref{rvm2-1}, note that by \eqref{Equation w}, $w$ satisfies
\begin{equation}\label{Equation w2}
-\Delta w+\nabla \theta=0 \quad \text{and} \quad \text{\rm div}(w)=0 \quad \text {in } 4B\cap\Omega
\end{equation}
 and $\frac{\partial w}{\partial \nu} =0$ on $I(x_0, 4r)$.
 Thus by the weak reverse H\"older condition \eqref{1.6}, we have
$$
\left\{ \fint_{D(x_0, 2r) }(|\nabla w|+|\theta|)  ^{p}  \right\}^{\frac{1}{p}} \le C\left\{ \fint_{D(x_0, 4r) }( |\nabla w| +|\theta|) ^{2}  \right\}^{\frac{1}{2}}.
$$
It follows that
$$
\begin{aligned}
\left\{ \fint_{D(x_0, 2r) }\left|R_{B}\right|^{p}  \right\}^{\frac{1}{p}}&
 \leq C\left\{ \fint_{D(x_0, 4r) }(|\nabla w| + |\theta|) ^2  \right\}^{\frac{1}{2}}\\
&\leqslant C\left\{ \fint_{D(x_0, 4r)}(|\nabla u| +|\phi|) ^{2}  \right\}^{\frac{1}{2}}
+C\left\{ \fint_{D(x_0, 4r)}( |\nabla v| +|\pi|)^{2}  \right\}^{\frac{1}{2}} \\
&\leqslant C\left\{ \fint_{D(x_0, 4r) }|H|^{2}  \right\}^{\frac{1}{2}} +C\left\{ \fint_{D(x_0, 8r) }|h|^{2}  \right\}^{\frac{1}{2}},
\end{aligned}
$$
which gives \eqref{rvm2-1}. Thus, by Theorem \ref{real2} and energy estimates,  we have
\begin{equation}\label{4.6}
\aligned
\left\{ \fint_{\Omega}  (|\nabla u|  + |\phi|) ^{q}  \right\}^{\frac{1}{q}}&\leq C\left\{ \fint_{\Omega} ( |\nabla u| +|\phi|  ) ^2  \right\}^{\frac{1}{2}}+C\left\{ \fint_{\Omega}\left|f\right|^{q}  \right\}^{\frac{1}{q}}\\
&\leq C\left\{ \fint_{\Omega}\left|f\right|^{q}  \right\}^{\frac{1}{q}}
\endaligned
\end{equation}
for $2< q< p$. Since the reverse H\"older condition \eqref{1.6} is self-improving,
the estimate \eqref{4.6}  holds for $2< q< \bar{p}$, where $\bar{p}>p$.
In particular, it holds for $q=p$.
As a result, we obtain \eqref{1.7}.
\end{proof}

Armed with Theorem \ref{1.5}, we  give the proof of Theorem \ref{W1p} for the case $p>2$ and $F,g=0$.

\begin{lemma}\label{p>2}
Let $\Omega$ be a bounded convex domain in $\mathbb{R}^d$.
 Let $(u,\phi)\in W^{1,2}(\Omega)\times L^2(\Omega)$ be a weak solution of
 \eqref{N-4-1} with $f\in L^p(\Omega; \mathbb{R}^{d\times d})$.
  Then there exists $\e>0$, depending on $\Omega$,
  such that for $d\ge 3$ and $2<p<\frac{2d}{d-2}+\e$,
\begin{equation}\label{w1p-1}
\|\nabla u\|_{L^p(\Omega)}+ \left\|\phi\right\|_{L^p(\Omega)} \leq C\|f\|_{L^p(\Omega)},
\end{equation}
where $C$ depends on $d, p$ and $\Omega$.
If $d=2$, the estimate \eqref{w1p-1} holds for any $2< p< \infty$.
\end{lemma}

\begin{proof}

In view of Lemma \ref{1.5}, it suffices  to show that the condition \eqref{1.6} holds for $2<p<\frac{2d}{d-2}+\e$ if $d\ge 3$,  and for
$2<p<\infty$ if $d=2$.
However, this was done in Corollary \ref{new-cor}.
 \end{proof}



\begin{thm}\label{p<2}
Let $\Omega$ be a bounded convex domain in $\mathbb{R}^d$.  Then there exists $\varepsilon>0$,
    depending only on $\Omega$, such that
    for any given $f\in C^\infty(\mathbb{R}^d; \mathbb{R}^{d\times d} )$ with $d\ge 3$ and
    \begin{equation}\label{range-1}
    \Big| \frac{1}{p}-\frac12\Big| < \frac{1}{d} +\e,
    \end{equation}
      the weak solution $(u,\phi)$ of
      the Neumann problem \eqref{N-4-1},
  satisfies
    \begin{equation} \label {w1p-2}
    \|\nabla u\|_{L^p(\Omega)}+\|\phi\|_{L^p(\Omega)} \leq C\|f\|_{L^p(\Omega)},
    \end{equation}
    where $C$ depends only on $ p$ and $\Omega$.
    If $d=2$, the estimate \eqref{w1p-2}  holds for $1<p< \infty$.
\end{thm}

\begin{proof}
The case $p>2$  is given by Lemma \ref{p>2}.
For $p<2$, we use a duality argument.
Let $(u, \phi)$ be a weak solution of \eqref{N-4-1}, where $f\in C^\infty (\mathbb{R}^d; \mathbb{R}^{d\times d})$.
Let $h\in C_0^\infty(\Omega; \mathbb{R}^{d\times d} )$ and $(v, \pi)$ be a weak solution of
  \begin{equation*}
  \left\{
  \aligned
			-\Delta v +\nabla \pi & ={\rm div} (h) & \quad & \text{ in }\Omega, \\
          \text{\rm div} (v) & = 0   & \quad &\text{ in }  \Omega, \\
			\frac{\partial v}{\partial \nu} & = 0 & \quad & \text{ on } \partial\Omega.
			\endaligned
			\right.
				\end{equation*}
Let $q=\frac{p}{p-1}$. Note that by Lemma \ref{p>2},   $\|\nabla v \|_{L^q(\Omega)} \le C \| h \|_{L^q(\Omega)}$.
In view of the weak formulation \eqref{weak-11} of variational solutions for  $(u,\phi)$ and $(v,\pi)$, one has
\begin{equation}\label{4.7}
\int_\Omega h\cdot \nabla u=-\int_\Omega \nabla u \cdot \nabla v=\int_\Omega f\cdot \nabla v.
\end{equation}
It follows  that
\begin{equation*}
\|\nabla u\|_{L^p(\Omega)}=\sup_{\|h\|_{L^q(\Omega)}\leq1}\left|\int_\Omega h\cdot \nabla u\,dx\right|\leq C\|f\|_{L^p(\Omega)}.
\end{equation*}
 By \eqref{wf-3}, we also obtain
$$
\| \phi \|_{L^p(\Omega)} \le C\left\{ \| \nabla u \|_{L^p(\Omega)} +  \| f\|_{L^p(\Omega)} \right\}
\le C \| f \|_{L^p(\Omega)}.
$$
\end{proof}

\subsection{The case of $f=0$ and $F=0$}

\begin{lemma}\label{dual_g}
	Let $\Omega$ be a bounded convex domain in $\rn{d}$. Let $(u,\phi)$ be a weak solution of
	\begin{equation*}
	\left\{
		\aligned
			-\Delta u +\nabla \phi & =0 & \quad & \text{in}\quad \Omega, \\
           \text{\rm div} (u) &= 0 &  \quad &\text{in}\quad \Omega, \\
           \frac{\partial u}{\partial \nu} & =g & \quad & \text{on}\quad \partial\Omega,
		\endaligned
		\right.
	\end{equation*}
	where $g\in B^{-1/p,p}(\partial \Omega; \mathbb{R}^d )\cap B^{-1/2, 2} (\partial\Omega; \mathbb{R}^d)$ , $\langle g, e_j\rangle
	_{B^{-1/2, 2} (\partial \Omega) \times B^{1/2, 2} (\partial\Omega)}=0$,
	and $p$ satisfies \eqref{range-1} for $d\geq 3$ and $1<p<\infty$ for $d=2$.
Then
	\begin{equation}\label{4.9}
		\|\nabla u\|_{L^p(\Omega)} + \| \phi \|_{L^p(\Omega)}  \leq C\|g\|_{B^{-1/p,p}(\partial \Omega)},
	\end{equation}
where $C$ depends only on $ p$ and $\Omega$.
\end{lemma}

\begin{proof}
  Let $f\in C_0^\infty(\Omega; \mathbb{R}^{d\times d} )$ and $(v, \pi)$ be a weak solution of
  \begin{equation*}
  \left\{
		\aligned
			-\Delta v +\nabla \pi & ={\rm div} (f) &  \quad  & \text{in}\quad \Omega, \\
           \text{\rm div}(v) &  =  0 & \quad & \text{in}\quad \Omega, \\
			\frac{\partial v}{\partial \nu} & =0 & \quad &  \text{on}\quad \partial\Omega.
		\endaligned
		\right.
	\end{equation*}
By Lemma \ref{p<2}, we know that $\|\nabla v\|_{L^p(\Omega)}\leq C\|f\|_{L^p(\Omega)}$. The weak formulations of variational solutions
for $(u, \phi)$ and $(v, \pi)$
imply that
$$
-\int_\Omega f\cdot \nabla u\,dx\ =\int_\Omega\nabla u \cdot \nabla v=\langle  g, v-\alpha  \rangle,
$$
where $\alpha =\fint_\Omega v$.
Note that
$$
\begin{aligned}
\left| \langle  g,  (v-\alpha )\rangle\right|&\leq \|g\|_{B^{-\frac{1}{p},p}(\partial\Omega)}\|v-\fint_\Omega v \|_{B^{\frac{1}{p},q}(\partial\Omega)}\\
&\leq C\|g\|_{B^{-\frac{1}{p},p}(\partial\Omega)}\|v-\fint_\Omega v \|_{W^{1,q}(\Omega)}\\
&\leq C\|g\|_{B^{-\frac{1}{p},p}(\partial\Omega)}\|\nabla v\|_{L^{q}(\Omega)}\\
&\leq C\|g\|_{B^{-\frac{1}{p},p}(\partial\Omega)}\|f\|_{L^{q}(\Omega)},
\end{aligned}
$$
where $1/p+1/q=1$. By  duality we obtain
$\|\nabla u \|_{L^p(\Omega)} \le C \| g \|_{B^{-1/p, p}(\partial\Omega)}$.
The desired estimate for $\phi$ follows from \eqref{wf-3}.
\end{proof}

\subsection{The case of $f=0$ and $g=0$}

\begin{lemma}\label{dual_F}
	Let $\Omega$ be a bounded convex domain in $\rn{d}$.
	 Let $(u,\phi)$ be a weak solution of
	\begin{equation*}
	\left\{
		\aligned
			-\Delta u +\nabla \phi & =F & \quad & \text{ in }  \ \Omega, \\
           \text{\rm div} (u) & = 0  & \quad & \text{ in } \Omega, \\
			\frac{\partial u}{\partial \nu} & =\frac{-1}{|\partial \Omega|}\int_\Omega F  &  \quad & \text{ on }\quad \partial\Omega,
		\endaligned
		\right.
	\end{equation*}
	where  $F\in C^\infty(\mathbb{R}^d; \mathbb{R}^d)$.
Then
	\begin{equation}\label{4.10}
		\|\nabla u\|_{L^p(\Omega)}  + \| \phi \|_{L^p(\Omega)} \leq C\|F\|_{L^{p}(\Omega)},
	\end{equation}
where $p$ satisfies \eqref{range-1} for $d\ge 3$ and $1< p< \infty$ for $d=2$.
\end{lemma}

\begin{proof}
  Let $(v, \pi)$ be a weak solution of
  \begin{equation*}
		\left\{
		\aligned
					-\Delta v +\nabla \pi & ={\rm div}  (f) & \quad  & \text{ in } \Omega, \\
          \text{\rm div}(v) & = 0  & \quad &  \text{ in }  \Omega, \\
		\frac{\partial v}{\partial \nu} & =0  & \quad &  \text{ on } \partial\Omega,
		\endaligned
		\right.
	\end{equation*}
where $f\in C_0^\infty(\Omega; \mathbb{R}^{d\times d} )$.
By Theorem \ref{p<2}, we know that $\|\nabla v\|_{L^q(\Omega)}\leq C\|f\|_{L^q(\Omega)}$,
where $q=\frac{p}{p-1}$.
  The weak formulations of variational solutions  $(u, \phi)$ and $(v, \pi)$ imply that
$$
\begin{aligned}
\int_\Omega f\cdot \nabla u  \,dx&=\int_{\partial\Omega}\left(\frac{1}{|\partial \Omega|}\int_\Omega F dx\right)\cdot v\,d\sigma
-\int_\Omega F \cdot v\\
& =\int_\Omega F \big(\fint_{\partial\Omega} v \big)dx-\int_\Omega  F \cdot v.
\end{aligned}
$$
It follows that
$$
\begin{aligned}
\Big| \int_\Omega f\cdot \nabla u  \,dx \Big|
&\leq \|F\|_{L^p(\Omega)}\|v-\fint_{\partial\Omega} v \|_{L^{q}(\Omega)}\\
&\leq C\|F\|_{L^p(\Omega)}\|\nabla v\|_{L^{q}(\Omega)}\\
&\leq C\|F\|_{L^p(\Omega)}\|f\|_{L^{q}(\Omega)}
\end{aligned}
$$
By duality we obtain $ \|\nabla u \|_{L^p(\Omega)} \le C \|  F \|_{L^p(\Omega)}$.
 As before, the bound for $\| \phi \|_{L^p(\Omega)}$ follows from \eqref{wf-3}.
\end{proof}

\subsection{ The proof of Theorem \ref{W1p}}

To prove the existence, we first assume that $F \in C^\infty(\mathbb{R}^d; \mathbb{R}^d)$,
$ f\in C^\infty (\mathbb{R}^d; \mathbb{R}^{d\times d})$ and $ g\in B^{-1/p, p}(\partial\Omega; \mathbb{R}^d) \cap B^{-1/2, 2}(\partial\Omega; \mathbb{R}^d)$
and  that $(F, f, g)$ satisfies the compatibility condition \eqref{comp}.
Let $(v,\pi), (w, \theta)$ and $(h, \eta)$ be weak solutions of
 \begin{equation*}
 \left\{
 \aligned
			- \Delta v +\nabla \pi &={\rm div} ( f ) &\quad &  \text{ in }  \Omega, \\
         \text{\rm div}(v)  & = 0    & \quad & \text{ in } \Omega, \\
			\frac{\partial v}{\partial \nu} & = -n\cdot f  & \quad & \text{ on }\partial\Omega,
 \endaligned
 \right.
	~~\text{and}~~~~
\left\{
\aligned
			-\Delta w +\nabla \theta & =F  & \quad &  \text{ in } \Omega, \\
           \text{\rm div} (w) &  = 0  & \quad &  \text{ in }  \Omega, \\
		\frac{\partial w}{\partial \nu} & =\frac{-1}{|\partial \Omega|}\int_\Omega F dx  & \quad  & \text{ on }  \partial\Omega,
		\endaligned
		\right.
	 \end{equation*}
and
\begin{equation*}
		\left\{
		\aligned
			-\Delta h + \nabla \eta & =0  & \quad &  \text{ in }  \Omega, \\
           \text{\rm div}(h) & = 0  & \quad &  \text{ in }  \Omega, \\
			\frac{\partial h}{\partial \nu} & =g+\frac{1}{|\partial \Omega|}\int_\Omega F dx & \quad &  \text{ on } \partial\Omega.
		\endaligned
		\right.
	\end{equation*}
	Then  $(u, \phi)=(v+ w + h, \pi + \theta + \eta)$ is a weak solution of \eqref{inhomogeneousStokes}.
By  Lemmas \ref{p>2}-\ref{dual_F} we obtain
$$
\begin{aligned}
& \|\nabla u\|_{L^p(\Omega)} + \|\phi\|_{L^p(\Omega)}\\
&\leq \|\nabla h\|_{L^p(\Omega)}+\|\nabla v\|_{L^{p}(\Omega)}+\|\nabla w\|_{L^{p}(\Omega)}
+ \| \pi \|_{L^p(\Omega)}
+ \| \theta\|_{L^p(\Omega)}
+ \|\eta\|_{L^p(\Omega)}\\
&\leq C\left\{\|g\|_{B^{-\frac{1}{p},p}(\partial\Omega)}+\|F\|_{L^p(\Omega)}+\|f\|_{L^{p}(\Omega)}\right\}.
\end{aligned}
$$
By a density argument we deduce that for any $f\in L^p(\Omega; \mathbb{R}^{d\times d} )$,
$F \in L^p(\Omega; \mathbb{R}^d)$ and $g\in B^{-1/p, p}(\partial\Omega; \mathbb{R}^d)$ that satisfy \eqref{comp},
there exist $u\in W^{1, p}(\Omega; \mathbb{R}^d)$ and $\phi \in L^p(\Omega)$ such that div$(u)=0$ in $\Omega$ and
\eqref{weak-NP} holds for any $\varphi \in C^\infty(\mathbb{R}^d; \mathbb{R}^d)$.
Moreover, the solution $(u, \phi)$ satisfies \eqref{1.3}.

To prove the uniqueness, suppose $u\in W^{1, p}(\Omega; \mathbb{R}^d)$ and $\phi\in L^p(\Omega)$ such that
div$(u)=0$ in $\Omega$ and
\begin{equation}\label{uni-1}
\int_\Omega \nabla u \cdot \nabla \varphi -\int_\Omega \phi \, \text{\rm div}(\varphi)=0
\end{equation}
for any $\varphi \in C^\infty(\mathbb{R}^d; \mathbb{R}^d)$.
By subtracting a constant, we may assume
\begin{equation}\label{comp-10}
\langle u, e_j \rangle_{B^{-1/p, p}(\partial\Omega) \times B^{1/p, p^\prime}(\partial\Omega)} =0
\end{equation}
 for $1\le j \le d$.
By a density argument, it follows that \eqref{uni-1} continues to hold for any $\varphi \in W^{1. q}(\Omega; \mathbb{R}^d)$, where $q=\frac{p}{p-1}$.
Since $|u |^{p-2}  u \in L^q(\Omega; \mathbb{R}^d)$,
by the existence established above,  there exists   $(v, \pi)\in W^{1, q}(\Omega)\times L^q(\Omega)$ such that
\begin{equation*}
 \left\{
 \aligned
			- \Delta v +\nabla \pi &=|u |^{p-2}  u  &\quad &  \text{ in }  \Omega, \\
         \text{\rm div}(v)  & = 0    & \quad & \text{ in } \Omega, \\
			\frac{\partial v}{\partial \nu} & =\alpha  & \quad & \text{ on }\partial\Omega,
 \endaligned
 \right.
\end{equation*}
where $\alpha =  \frac{-1}{|\partial\Omega|} \int_\Omega |u|^{p-2} u \, dx$.
Thus,
$$
\int_\Omega \nabla v \cdot \nabla u =\int_\Omega |u|^p ,
$$
where we have used the assumption \eqref{comp-10}.
However, by letting $\varphi=v$ in \eqref{uni-1}, we obtain $\int_\Omega \nabla u \cdot \nabla v=0$.
As a result, $\int_\Omega | u|^p=0$. We obtain $u=0$ in $\Omega$.
By using  \eqref{uni-1}, we also  deduce that  $\phi=0$ in $\Omega$.

\newpage
\bibliography{Stokes}

\small
\noindent\textsc{Jun Geng, School of Mathematics and Statistics,
Lanzhou University,
Lanzhou 730000, Gansu, P.R. China.}\\
\emph{E-mail address}: \texttt{gengjun@lzu.edu.cn} \\

\small
\noindent\textsc{Zhongwei Shen, Department of Mathematics, University of Kentucky, Lexington, Kentucky 40506,
USA.}\\
\emph{E-mail address}: \texttt{zshen2@uky.edu} \\

\end{document}